\documentclass[smallextended]{amsart}       % onecolumn (second format)
\usepackage{amsmath,amssymb,amsfonts}

\newtheorem{thm}{Theorem}
\newtheorem{lem}{Lemma}
\newtheorem{prop}{Proposition}
\newtheorem{res}{Corollary}
\newtheorem{rem}{Remark}

\newtheorem{definition}{Definition}
\newtheorem{ex}{Example}
\newcommand{\ve}{\varepsilon}
\newcommand{\eps}{\varepsilon}

\renewcommand{\P}{\mathsf P}
\newcommand{\E}{\mathsf E}
\newcommand{\Ef}{\mathcal E}

\newcommand{\F}{\mathcal F}

\newcommand{\bx}{\mathbf{x}}
\newcommand{\R}{\mathbb R}
\renewcommand{\Re}{\mathbb R}
\newcommand{\df}{\mathrm d}
\newcommand{\prt}{\partial}
\newcommand{\mc}[1]{\mathcal {#1}}
\newcommand{\dif}{{\mathrm D}}
\newcommand{\dom}{{\rm dom}(\dif)}
\newcommand{\1}{1\!\!\hbox{{\rm I}}}
\newcommand{\be}{\begin{equation}}
\newcommand{\ee}{\end{equation}}
\def\institute#1{\gdef\@institute{#1}}

\begin{document}

\title[Malliavin calculus approach to statistical inference for
SDE's]{Malliavin calculus approach to statistical inference for
L\'{e}vy driven SDE's}

\author{D. O. Ivanenko \and  A. M. Kulik}

\institute{D. O. Ivanenko \at
              Kyiv National Taras Shevchenko  University, Volodymyrska, 64, Kyiv, 01033, Ukraine\\
              \email{ida@univ.kiev.ua}
              \and
           A. M. Kulik \at Institute of Mathematics, Ukrai\-ni\-an National Academy of Sciences, 01601 Tereshchenkivska, 3, Kyiv, Ukraine \\
           \email{kulik@imath.kiev.ua}
}

\date{Received: date / Accepted: date}

\maketitle

\begin{abstract}
By means of the Malliavin calculus, integral representations for
the likelihood function and for the derivative of the
log-likelihood function are given for a model based on discrete
time observations of the solution to equation $
    \df X_t=a_\theta(X_t)\df t +\df Z_t
$ with a  L\'evy process $Z$. Using these representations,
regularity of the  statistical experiment and  the Cramer-Rao
inequality are proved. 
\end{abstract}

\section{Introduction}
Consider stochastic equation of the form
\begin{equation}\label{eq1}
    \df X_t=a_\theta(X_t)\df t +\df Z_t,
\end{equation}
where $Z$ is a one-dimensional L\'evy process,
$a:\Theta\times\R\to\R$ is a measurable function,
$\Theta\subset\R$ is a parametric set. The main objective of our
study is the statistical inference of the unknown parameter
$\theta$ given the observations of the  solution to this equation
at a discrete time set.

The likelihood function in the above model is highly implicit. In
this paper, we develop an approach which makes it possible to
control the properties of the likelihood and log-likelihood
functions only in the terms  of the objects involved in the model:
the function $a_\theta(x)$, its derivatives,  and the L\'evy
measure of the L\'evy process $Z$. This approach is based on an
appropriate version of the Malliavin calculus for  a Poisson point
measure.

The Malliavin calculus,  developed first as a tool for proving
existence and smoothness of distribution densities, appears to be
very  efficient in a study of sensitivities of expectations w.r.t.
parameters. This field of applications, motivated by the analysis
of volatilities in the models of financial mathematics, comes back
to \cite{finance} and  was studied intensively during the last
years. This technique has  natural extensions to statistical
problems. In \cite{Gobet1}, \cite{Gobet2}, for discretely observed
diffusion models,  a Malliavin-type  integral representation of
the derivative of the log-likelihood ratio w.r.t. parameter is
given, and then is used as a key tool in the proof of the LAN
(LAMN) property of the model. In the recent paper \cite{Corc},
several versions of Malliavin calculus-based  sensitivity analysis
on the Wiener-Poisson probability space was developed, with
applications to evaluation of the Cramer-Rao inequality and to a
study of asymptotic properties of  MLE for discretely observed
diffusion processes.

We are mainly concentrated on the study of equation \eqref{eq1},
where $Z$ is a L\'evy process without a diffusion component. We
develop a particular version  of the Malliavin calculus for
Poisson point measures from \cite{Bismut_jumps},
\cite{Bichteler_Grav_Jacod} (see also more recent papers
\cite{BC}, \cite{BD} and references therein), which is  convenient
for the purposes of the further sensitivity analysis. We give
integral representations for the likelihood function and for the
derivative of the log-likelihood function w.r.t. parameter. These
representations  are used then as the key ingredient in the proof
of the regularity of a statistical experiment generated by
discrete time set observations of the solution to (1), and
consequent evaluation of the Cramer-Rao inequality. These
representations also give a tool for the further asymptotical
analysis of the properties on the model when the size of the
sample tends to $\infty$; see forthcoming papers \cite{Ivanenko}
and \cite{IvanenkoKulik}, addressed to the LAN property of the
model and to the asymptotic efficiency of the MLE, respectively.

For simplicity reasons, here we restrict ourselves by the case of
both observations $X_t$ and parameter $\theta$ being
one-dimensional, and postpone the study of the multidimensional
case for a further research.

\section{Notation, assumptions, and  main results}

\subsection{Notation and assumptions}
Let  $Z$ be  a L\'evy process without a diffusion component; that is,
$$
Z_t=ct+\int_0^t\int_{|u|>1}u\nu(\df s, \df u)+
\int_0^t\int_{|u|\leq 1}u\tilde\nu(\df s, \df u),
$$
where $\nu$ is a Poisson point measure with the intensity measure
$\df s\ \mu(\df u)$, and  $\tilde \nu (\df s, \df u)=\nu(\df s,
\df u)-\df s\ \mu(\df u)$ is respective compensated Poisson
measure. In the sequel, we assume the L\'evy measure $\mu$ to
satisfy the following:

\textbf{H.} (i) for some $\kappa>0$,
$$
\int_{|u|\geq 1}u^{2+\kappa}\mu(du)<\infty;
$$

(ii) for some $u_0>0$, the restriction of $\mu$ on $[-u_0, u_0]$
has a positive density $\sigma\in C^2\left(\left[-u_0,0\right)\cup
\left(0, u_0\right]\right)$;

(iii) there exists $C_0$ such that
$$
|\sigma'(u)|\leq C_0|u|^{-1}\sigma(u),\quad |\sigma''(u)|\leq
C_0u^{-2}\sigma(u),\quad |u|\in (0, u_0];
$$

(iv)
$$
\left(\log {1\over \eps}\right)^{-1}\mu\Big(\{u:|u|\geq
\eps\}\Big)\to \infty,\quad \eps\to 0.
$$
One particularly important class of  L\'evy processes satisfying
\textbf{H}  consists of \emph{tempered $\alpha$-stable processes}
(see \cite{Ros1}), which arise naturally in models of turbulence
\cite{novikov}, economical models of stochastic volatility
\cite{carr}, etc.

Without loss of generality, $\Theta$ is assumed to be a finite
open interval on $\Re$. For a given $\theta\in \Theta$, assuming
that the drift term $a_\theta$ satisfies the standard local
Lipschitz and linear growth conditions,  Eq. (\ref{eq1}) uniquely
defines a Markov process $X$. We denote by $\P_x^\theta$ the
distribution of this process in $\mathbb{D}([0, \infty))$ with
$X_0 =x$,  and  by $\E_x^\theta$ the expectation w.r.t. this
distribution. Respective finite-dimensional distribution for given
time moments $t_1<\dots<t_n$ is denoted by
$\P_{x,\{t_k\}_{k=1}^n}^\theta$. On the other hand, solution $X$
to Eq. (\ref{eq1}) is a random function defined  on the same
probability space $(\Omega, \F, \P)$ with the process $Z$, which
depends additionally on the parameter $\theta$ and the initial
value $x=X(0)$. We do not indicate this dependence in the
notation, i.e. write $X_t$ instead of e.g. $X^\theta_{x,t}$, but
it will be important in the sequel that, under certain conditions,
$X_t$ is $L_2$-differentiable w.r.t. $\theta$ and is
$L_2$-continuous w.r.t $(t,x,\theta)$.

In the sequel we will show that, under appropriate conditions,
Markov process $X$ admits a transition probability density
$p^\theta_t(x,y)$ w.r.t. Lebesgue measure, which is continuous
w.r.t. $(t,x,y)\in (0,\infty)\times \Re\times \Re$. Then  (see
\cite{bridges}), for every $t>0, x,y\in \Re$ such that
\be\label{pnon} p^\theta_t(x,y)>0,\ee  there exists a weak limit
in $\mathbb{D}([0, t])$
$$
\P^{t, \theta}_{x,y}=\lim_{\eps\to 0}\P^\theta_x\Big(\cdot\Big||X_t-y|\leq \eps\Big),
$$
which can be interpreted naturally as a \emph{bridge} of the
process $X$ started at $x$ and conditioned to arrive to $y$ at
time $t$. We denote by $\E^{t, \theta}_{x,y}$ the expectation
w.r.t. $\P^{t, \theta}_{x,y}$.

In what follows, $C$ denotes a constant which is not specified
explicitly and may vary from place to place. By
$C^{k,m}(\Re\times\Theta), k,m\geq 0$ we denote the class of
functions $f:\Re\times\Theta\to \Re$ which has continuous
derivatives
$$
\prt^{i+j}_{x\dots x \theta \dots \theta}f :={\prt^i\over\prt
x^i}{\prt^j\over \prt \theta^j}f, \quad i\leq k, \quad j\leq m.
$$

\subsection{Main results: formulation}
Here we formulate two main theorems of this paper.  The first one
concerns the local properties of the transition probabilities of
the Markov process $X$. The functionals $\Xi_t,\ \Xi^{1}_{t}$,
involved in its formulation, will be introduced explicitly in the
proof below; see formulae (\ref{Xi}) and (\ref{Xi_1}).

\begin{thm}\label{mainthm1} I. Let $a\in C^{2,0}(\Re\times\Theta)$ with bounded derivatives
$\prt_x a_\theta, \prt^2_{x x}a_\theta$.

Then   the Markov process $X$ defined by (\ref{eq1}) has a
transition probability density $p_t^\theta$ w.r.t. the Lebesgue
measure, which has an integral representation \be\label{irp}
p_t^\theta(x,y)=\E_x^\theta\left[\Xi_t \1_{X_t>y}\right], \quad
t>0, \quad x,y\in \Re. \ee The function $p_t^\theta(x,y)$ is
continuous  w.r.t. $(t,x,y, \theta)\in (0, \infty)\times \Re\times
\Re\times \Theta$.

II. Let $a\in C^{3,1} (\Re\times\Theta)$ have bounded derivatives
$\prt_xa$, $\prt^2_{xx}a$, $\prt^2_{x\theta}a$,
$\prt^3_{xx\theta}a$, $\prt^4_{xxx\theta}a$ and \be\label{lin_gr}
|a_\theta(x)|+|\partial_{\theta} a_\theta(x)|\leq
    C(1+|x|), \quad \theta\in \Theta, \quad x\in \Re.
\ee

Then the transition probability density has a derivative
$\prt_\theta p_t^\theta(x,y),$ which is continuous  w.r.t.
$(t,x,y, \theta)\in (0, \infty)\times \Re\times \Re\times \Theta$.

III. Under the conditions of statement II, one has
\be\label{d_theta_rep}
\prt_\theta p_t^\theta(x,y)= g_t^\theta(x,y) p_t^\theta(x,y)
\ee
with
\be\label{g}
g_t^\theta(x,y)=\begin{cases}\E^{t,\theta}_{x,y}\Xi_t^1, &  p_t^\theta(x,y)>0,\\
0,&\hbox{otherwise}.\end{cases}
\ee
\end{thm}

\begin{rem} By statements II and III,  the logarithm of  the transition
probability density has a continuous derivative w.r.t. $\theta$ on
the open subset of $(0,\infty)\times \Re\times \Re\times \Theta$
defined by inequality (\ref{pnon}) and, on this subset, admits the
integral representation \be\label{log} \prt_\theta \log
p_t^\theta(x,y)=\E^{t,\theta}_{x,y}\Xi_t^1. \ee
\end{rem}

The second theorem concerns the basic properties of the
statistical experiment \be\label{stat_exp} \Big(\Re^n,
\mathcal{B}(\Re^n), \P_{x,\{t_k\}_{k=1}^n}^\theta, \theta\in
\Theta\Big), \ee generated by observations of the Markov process
$X$ with fixed $X_0=x$ at time moments $t_1<\dots<t_n$; we refer
to \cite{IKh} for the notation and terminology.  Recall that a
statistical experiment $(\mathcal{X}, \mathcal{U}, P^\theta,
\theta\in \Theta)$ is called \emph{regular}, if $\df
P^\theta=p^\theta \df \lambda$ for some $\sigma$-finite measure
$\lambda$, and
\begin{itemize}
\item[(a)] the function $\theta\mapsto p^\theta(\bx)$ is
continuous  for $\lambda$-a.a. $\bx\in \mathcal{X}$;

\item[(b)] the function  $\theta\mapsto \sqrt{p^\theta}\in
L_2(\mathcal{X}, \lambda)$ is differentiable;

\item[(c)] the function $\theta\mapsto \prt_\theta
\sqrt{p^\theta}\in L_2(\mathcal{X}, \lambda)$ is continuous (the
derivative is understood in the $L_2(\mathcal{X}, \lambda)$
sense).
\end{itemize}
For a regular statistical experiment with $\Theta\subset \Re^1$,
respective \emph{Fisher information} is defined as
$$
I(\theta)=4\int_{\mathcal{X}}
\left(\prt_\theta\sqrt{p^\theta}\right)^2\df\lambda.
$$

\begin{thm}\label{mainthm2}  Let conditions of statement II of Theorem \ref{mainthm1}
hold true and $x\in \Re, n\in \mathbb{N}$,  $0<t_1<\dots<t_n$ be fixed.

Then the statistical experiment (\ref{stat_exp}) is regular.
Respective Fisher  information equals
$$
I(\theta)=\sum_{k=1}^n\E_x^\theta
\Big(g_{t_k-t_{k-1}}^\theta(X_{t_{k-1}}, X_{t_k})\Big)^2,
$$
%\int_{\Re^n}\Big(g_h^\theta(x_{k-1}, x_k)\Big)^2 \prod_{j=1}^n p^\theta_h(x_{j-1}, x_j)\,dx_1\dots dx_n,
where $t_0:=0$.

\end{thm}

\begin{rem} For the statistical experiment (\ref{stat_exp}) one has $\mathcal{X}=\Re^n$,
and the natural choice of  $\lambda$ is the Lebesgue measure. Then
by Theorem \ref{mainthm1} \be\label{rep_p}
p^\theta(\bx)=\prod_{k=1}^np_{t_k-t_{k-1}}^\theta(x_{k-1}, x_k),
\quad \bx=(x_1, \dots, x_n), \ee where  $x_0:=x$, and there exists
a point-wise derivative \be\label{rep_dp} \prt_\theta
p^\theta(\bx)=p^\theta(\bx) g^\theta(\bx),\quad
g^\theta(\bx):=\sum_{k=1}^ng_{t_k-t_{k-1}}^\theta(x_{k-1}, x_k),
\quad \bx=(x_1, \dots, x_n). \ee In particular, one can interpret
(\ref{irp}) and (\ref{g})  as integral representations for the
log-likelihood function and the derivative of the likelihood
function in a one-point observation model.
\end{rem}

Combining Theorem \ref{mainthm2} and Theorem I.7.3 in \cite{IKh},
we obtain the following version of the Cramer-Rao inequality.

\begin{res}  Let conditions of statement II of Theorem \ref{mainthm1}
hold true and $x\in \Re, n\in \mathbb{N}$,  $0<t_1<\dots<t_n$ be fixed.  Assume that
$$
I(\theta)>0
$$
and $T:\Re^n\to \Re$ is a Borel measurable function such that the function
$$
\theta\mapsto \E_x^\theta T^2(X_{t_1}, \dots, X_{t_n})
$$
is locally bounded.

Then the bias
$$
d(\theta)=\E_x^\theta T(X_{t_1}, \dots, X_{t_n})-\theta
$$
is differentiable, and
$$
\E_x^\theta \Big(T(X_{t_1}, \dots, X_{t_n})-\theta\Big)^2\geq
{(1+\prt_\theta d(\theta))^2\over I(\theta)}+d^2(\theta).
$$

\end{res}

\subsection{Main results: discussion} Let us emphasize one particularly important
property of our model, which makes a substantial difference with
those studied in \cite{Gobet1}, \cite{Gobet2}, \cite{Corc}. In the
diffusive models studied in \cite{Gobet1}, \cite{Gobet2}, the
likelihood function is positive, hence the  log-likelihood
function is a $C^1$ function w.r.t. $\theta$. The approach of
\cite{Corc} requires, among others,  the following structural
assumption: \be\label{i} \hbox{the support of the density
$p^\theta$ does not depend on $\theta$}, \ee which also  makes it
possible to obtain the $C^1$ log-likelihood function by
considering  respective support set  as a state space
$\mathcal{X}$. However, as one can  see from the  example below,
in the  context of  Eq. (\ref{eq1}), the assumption (\ref{i})
would restrict the model  substantially.

\begin{ex}\label{ex1} Let  $Z$ be a tempered $\alpha$-stable with $\alpha\in (0,1)$,
which has positive jumps, only; that is,
$$
\mu(\df u)=r(u)u^{-\alpha-1}\1_{u>0}\df u
$$
with some (smooth enough) $r$ such that $r(u)=\mathrm{const}>0$
in a neighborhood of $u=0$, and $r(u)\to 0$ (rapidly enough) as
$u\to \infty$. Then by the support theorem from
\cite{t_simon_support}, the topological support of
$P_t^\theta(x,\df y)$ equals $[y_t^\theta(x), \infty)$, where
$y_t^\theta(x)$ is the value at time moment $s=t$ of the solution
to the Cauchy problem
$$
y'(s)=a_\theta(y(s)), \quad y(0)=x.
$$
Generically, $y_t^\theta(x)$ depends on $\theta$; for instance,
for $a_\theta(x)=\theta x$ one has $y_t^\theta(x)=xe^{t\theta}$.
Because the topological support of $P_t^\theta(x,\df y)$ is the
closure of the support of the transition probability density
$p^\theta_t(x,y)$, this indicates that in this case  (\ref{i})
fails.
\end{ex}

This observation mainly motivates the particular form of our
approach: because we would like to exclude the assumption
(\ref{i}) completely,  we do not  rely on the path-wise regularity
of the log-likelihood function. Instead of that, we prove  that
our model is regular. Regularity of the experiment yields the
Cramer-Rao inequality, and, which is maybe even more important, is
a natural pre-requisite for the Ibragimov-Khasminskii's version of
the Hayek-Le Cam approach to the study of the asymptotic properties of a model (\cite{IKh}, Chapter II.3 and Chapter
III.1). In our forthcoming papers \cite{Ivanenko} and
\cite{IvanenkoKulik}, we use both the integral representations
(\ref{irp}), (\ref{g})  and the regularity of the model to prove
the LAN property of the model and  the asymptotic efficiency of
the MLE.

\section{Malliavin calculus  for Poisson  random measures}\label{seq3}

Typically, a Malliavin calculus-based sensitivity analysis
requires a pair of a derivation operator $\dif$  and an adjoint
operator $\delta=\dif^*$ to be defined on the probability space
under the consideration. Below we outline such a construction,
based on perturbations of ``jump amplitudes'',  which  is well
known in the field, goes back to  \cite{Bismut_jumps},
\cite{Bichteler_Grav_Jacod}, and has various modifications; see
e.g. \cite{BC}, \cite{BD}, and references therein. To keep the
exposition self-sufficient and transparent, we  explain the main
components of this construction; in addition, we specially modify
it in order to provide the integral representations,  involved
into the  further statistical applications, in as explicit form as
it is possible.
\subsection{Perturbations of Poisson  random measures and associated differential operators}

Let $\varrho:\Re\to \Re^+$ be a $C^2$-function with bounded
derivative.  Denote by $Q_c(x),\ c\in \Re$  the value at the time
moment $s=c$ of the solution to  Cauchy problem
$$
q'(s)=\varrho(q(s)), \quad q(0)=x.
$$
Then $\{Q_c, c\in \Re\}$ is a group of transformations of $\Re$,
and $\prt_c Q_c(x)|_{c=0}=\varrho(x)$.

Denote by $\mc{O}$ the space of \emph{locally finite
configurations}  in $\Re^+\times(\Re\setminus\{0\})$; that is, the
family of all sets $\varpi\subset \Re^+\times(\Re\setminus\{0\})$
such that for any $\eps>0, R>0, T>0$ the set
$$
\varpi\cap \Big([0, T]\times \{u:\eps<|u|<R\}\Big)
$$
is finite. This space is naturally endowed by the \emph{vague
topology};  that is, the minimal topology w.r.t. which any  map of
the form
$$
\varpi\mapsto \sum_{(\tau,u)\in \varpi}f(\tau, u)
$$
with a continuous $f: \Re^+\times(\Re\setminus\{0\})$, which  is
supported by some set of the form $[0, T]\times \{u:\eps<|u|<R\}$,
is continuous.   This is the natural state space when the random
point measure $\nu$ is considered as a random element; denote by
$\P_\nu$  the distribution of $\nu$ in $(\mc{O}, \mc{B}(\mc{O}))$.
In what follows, we  identify the initial probability space
$(\Omega, \F, \P)$ with $(\mc{O}, \mc{B}(\mc{O}), \P_\nu)$, and
assume $\nu(\omega)=\omega$. Under this convention, which does not
restrict generality, every $\omega\in \mc{O}$ is a locally finite
collection of points $(\tau, u)$, where $\tau\in \Re^+$ is the
``jump time'', and $u\in \Re\setminus\{0\}$ is the ``jump
amplitude''.

 For  given $T>0$ and $\varrho$, define the group $\{\mc{Q}_c, c\in \Re\}$
 of transformations of the configuration space $\mc{O}$ in the following way.
 Transformation  $\mc{Q}_c$ maps a  configuration  $\omega$ into the
 collection of points of the form
$$
\begin{cases} (\tau, Q_c (u)), & (\tau, u)\in \omega\hbox{ is such that $\tau\leq T$};\\
 (\tau, u), & (\tau, u)\in \omega\hbox{ is such that $\tau>T$}.
 \end{cases}
$$
Then $\mc{Q}_c$ transforms  $\nu$ into the Poisson point measure $\nu_c$,
$$\begin{aligned}\nu_{c}(A\times B)=\nu\left(\left(A\cap[0,T]\right)\times
Q_c^{-1}(B)\right)+&\nu\left(\left(A\cap\overline{[0,T]}\right)\times
B\right),\\&\quad \quad   A\in\mc{B}(\R_+), B\in\mc{B}(\R),\end{aligned}$$  and
the intensity measure for $\nu_c$
has the form \be\label{muc} \1_{s\leq T}\df s\, [\mu\circ
Q_c^{-1}](\df u)+\1_{s>T}\df s\, \mu(\df u).\ee

Fix $u_1\in (0, u_0)$, where $u_0$  comes from \textbf{H} (ii). In
what  follows, we choose  the function  $\varrho$ involved in the
definition of $Q_c, c\in \Re$ in such a way that
$$
\varrho(u)=\begin{cases} u^{2},&|u|\leq u_1;\\
0,&|u|\geq u_0\end{cases}.
$$
Then  the intensity measure (\ref{muc}) has the density w.r.t.
$\df s\ \mu(\df u)$ equal to \be\label{pct} m_{c,T}(s,u)=\1_{s\leq
T} m_{c}(u)+\1_{s>T}, \ee where
$$
m_{c}(u)={\df[\mu\circ Q_c]^{-1}\over \df\mu}(u)={1\over
R_c(Q_c^{-1}(u))} {\sigma(Q_c^{-1}(u))\over
\sigma(u)}
$$
with
$$
R_c(x):=\prt_xQ_c(x)=\exp\left(\int_0^c\varrho'(Q_s(x))\, \df
s\right).
$$
By the construction, $Q_c(u)\equiv u, c\in \Re$ if $|u|\geq u_0$.
On the other hand, for any given $c\in \Re$ there exists $u(c)>0$
s.t.
$$
Q_s(u)=\left({1\over u}-s\right)^{-1}={u\over 1-us}, \quad |s|\leq |c|, \quad |u|\leq u(c).
$$
Therefore there exists $\hat u(c)\in (0, u(c))$ s.t. \be\label{21}
{Q_s(u)\over u}\in \left[\frac{1}{2}, 2\right], \quad |s|\leq |c|,
\quad |u|\leq \hat u(c). \ee Because $Q^{-1}_c=Q_{-c}$, this
yields immediately that
$$
R_c(Q^{-1}_c(u))=O(|u|), \quad u\to 0.
$$
Using  (\ref{21}) and \textbf{H} (iii), we get for $|u|\leq \hat
u(c)$
$$
\left|\sigma\left(Q_c^{-1}(u)\right)-\sigma(u)\right|=
\left|\int_0^c\sigma'(Q_s^{-1}(u))Q_s^{-1}(u)^2\, \df s\right|\leq
2|u| \int_0^{|c|}\sigma\left(Q_s^{-1}(u)\right)\, \df s.
$$
It is straightforward to deduce from \textbf{H}(iii) that, for some $u_2>0$ and $K>1$,
$$
\sigma(\gamma u)\leq K\sigma(u), \quad \gamma\in \left[1/2,
2\right], \quad |u|\leq u_2.
$$
Summarizing all the above, we conclude that for a given $c$ the
function $\log m_{c} $ is continuous, vanishes when $|u|\geq u_0$,
and satisfies
$$
\log m_{c}(u)=O(|u|), \quad u\to 0.
$$
Therefore, one has $$
 \int\limits_{|\log m_{c,T}|\geq\log 2}\left|
 1-m_{c,T}(s,u)\right|\, \df s\mu(\df u)+
 \int\limits_{\R^+\times \Re}\frac{\log^2 m_{c,T}(s,u)}{1+\log^2 m_{c,T}(s,u)}\df s\mu(\df
 u)<\infty.
$$
Applying  Skorokhod's criterion for absolute continuity of the
laws of  Poisson point measures \cite{skorohod}, we arrive at
following.

\begin{prop} The distribution $\P_{\nu_c}$ of $\nu_c$ in $(\mc{O}, \mc{B}(\mc{O}))$
is absolutely continuous w.r.t. $\P_\nu$, and
\begin{multline}\label{pi_c}
 \kappa_c:=\frac{\df\P_{\nu_c}}{\df\P_\nu}(\nu)=
 \exp\left\{ \int_0^T\int_{\R}\log m_{c}(u)\tilde \nu(\df s, \df u)\right.\\
 \left.+T\int_{\R}\left(
 1-m_{c}(u)+\log m_{c}(u)\right)\mu(\df
 u)
\right\}.
\end{multline}
\end{prop}

Consequently, the map $\mc{Q}_c:\Omega\to\Omega$ generates the map
of   $L_0(\Omega, \F, \P)$ into itself in the following way (we
keep the same symbol $\mc{Q}_c$ for this map):
$$
\mc{Q}_cF(\omega)= F(\mc{Q}_c\omega), \quad F\in L_0(\Omega, \F, \P).
$$

Straightforward computation shows that for every $u\not=0$
$$
\prt_cm_{c}(u)|_{c=0}=-{(\sigma(u)\varrho(u))'\over \sigma(u)}=:\chi(u).
$$
In addition,
\be\label{22}
\int_{\Re}\left({m_{c}(u)-1\over c}-\chi(u)\right)^2\mu(\df u)\to
0, \quad c\to 0.
\ee
Because $m_{c}(u)=1, |u|\geq u_0$,  the latter relation and
(\ref{pi_c}) yield \be\label{kappa} {\kappa_c(u)-1\over c}\to
\int_0^T\int_{\Re}\chi(u)\tilde \nu(\df s,\df u) \hbox{ in  }
L_2(\Omega, \F, \P), \quad c\to 0. \ee The proofs  of (\ref{22}) and (\ref{kappa}) are
straightforward but cumbersome, and therefore are omitted.

\begin{definition}\label{def1}. A functional  $F\in L_2(\Omega, \F, \P)$
is called \emph{stochastically differentiable}, if there exists an
$L_2(\Omega, \F, \P)$-limit \begin{equation}\label{dif} \hat\dif
F=\lim_{c\to 0}{1\over c}\Big(\mc{Q}_cF-F\Big).
\end{equation}
The closure $\dif$ of the operator $\hat \dif$ defined by
(\ref{dif}) is called the \emph{stochastic derivative}.  The
adjoint operator $\delta=\dif^*$ is called the \emph{divergence
operator} or the \emph{extended stochastic integral}.
\end{definition}

\begin{rem} By (\ref{chain}) and (\ref{DZ})  below,  $\dom$ is dense in
$L_2(\Omega, \F, \P)$, hence $\delta$ is well defined. In
addition, by statement 3 of Proposition \ref{lem01} below
$\mathrm{dom}(\delta)$ is dense in $L_2(\Omega, \F, \P)$, hence
$\hat \dif$ is closable.
 The operator $\delta$ itself is closed as an adjoint one; e.g. Theorem VIII.1
 in \cite{gorod}.
\end{rem}

The following proposition collects the main properties of the operators $\dif,\delta$.

\begin{prop}\label{lem01} 1. Let $\varphi\in C^1(\Re^d, \Re)$ have bounded
derivatives and  $F_k\in\dom$, $k=\overline{1,d}$.

Then
$\varphi(F_1,\dots,F_d)\in\dom$ and
\be\label{chain}\dif\left[
\varphi(F_1,\dots,F_d)\right]=\sum\limits_{k=1}^d[\partial_{x_k}\varphi](F_1,\dots,F_d)\dif
F_k.\ee

2. The constant function $1$ belongs to $\mathrm{dom}(\delta)$ and
\be\label{delta_1} \delta(1)=\int_0^T\int_{\Re}\chi(u)\tilde
\nu(\df s,\df u). \ee

3. Let  $G\in\dom$ and
\begin{equation}\label{2}
    \E\left(\delta(1)G\right)^2<\infty.
\end{equation}

Then $G\in\mathrm{dom}(\delta)$ and $\delta(G)=\delta(1)G-\dif G. $
\end{prop}

\begin{proof} 1. It is sufficient to consider $F_k, k=\overline{1,d}$
which satisfy (\ref{dif}). Then the fraction \be\label{fra}
{1\over c}(\mc{Q}_c\varphi(F_1,\dots,F_d)-\varphi(F_1,\dots,F_d))
\ee converges in probability to the right hand side of
(\ref{chain}).  Its square is dominated by
$$
\sup_{x}\|\nabla \varphi(x)\|^2\sum_{k=1}^d \left({\mc{Q}_c
F_k-F_k\over c}\right)^2,
$$
and hence is uniformly integrable. Consequently,  (\ref{fra})
converges in $L_2(\Omega, \F, \P)$.

2. For any $F$ satisfying (\ref{dif}) we have by (\ref{kappa})
$$
\E F \left(\int_0^T\int_{\Re}\chi(u)\tilde \nu(\df s,\df
u)\right)=\lim_{c\to 0}{1\over c}\E F(\kappa_c-1)=\lim_{c\to
0}{1\over c}\E(\mc{Q}_cF-F)=\E(\dif F),
$$
which gives by the definition of $\delta=\dif^*$ that
$\delta(1)=\int_0^T\int_{\Re}\chi(u)\tilde \nu(\df s,\df u)$.

3. For \emph{bounded} $F,G\in \dom$ one has by (\ref{chain})  that
$FG\in \dom$ and $\dif(FG)=F\dif G+G\dif F$. Then by statement 2
we have \be\label{ibp} \E G\dif F=\E F\Big(\delta(1)G-\dif G\Big).
\ee For arbitrary $F\in \dom$, using (\ref{chain}), one can choose
a sequence of bounded $F_n\in \dom$ such that $F_n\to F$ and $\dif
F_n\to \dif F$ in $L_2(\Omega, \F, \P)$. This proves (\ref{ibp})
for arbitrary $F\in \dom$, and yields the required statement under
the additional assumption that $G$ is bounded. Approximating $G$
by bounded $G_n\in \dom$ and using that $\delta$ is a closed
operator completes the proof.
\end{proof}

\subsection{Differential properties of the solution to \eqref{eq1}}\label{seq4}

Denote $Z_t^c=\mc{Q}_cZ_t$. It can be seen straightforwardly that
\be\label{DZ}{1\over c}(Z_t^c-Z_t)\to \dif
Z_t:=\int_0^t\int_{\Re}\varrho(u)\nu(\df s,\df u) \hbox{ in
} L_2(\Omega, \F, \P) \ee uniformly by $t\in [0,T]$ for
every $T$. Then one can consider $X_t^c=\mc{Q}_cX_t$ as the
solution to the following perturbed SDE:
\begin{equation}\label{eq_c}
    \df X_t^{c}=a_\theta(X_t^{c})\df t +\df Z_t^c.
\end{equation}
Applying Theorem II.2.8.5 \cite{scorohod2}, under conditions of
statement I, Theorem \ref{mainthm1}  we get that for any fixed
$\theta\in \Theta$ and initial value $x\in \Re$ the family $X^c_t$ is differentiable in $L_2$; that is,
\be\label{DX}{1\over c}(X_t^c-X_t)\to \dif X_t \hbox{ in }
L_2(\Omega, \F, \P), \quad c\to 0. \ee Clearly, the derivative in the right hand side of (\ref{DX})  is just the stochastic derivative of $X_t$; see Definition \ref{def1}. Moreover, convergence (\ref{DX}) holds true uniformly by $t\in [0,T]$
for every $T$. The process $Y_t:= \dif X_t$ satisfies the linear
SDE
$$
 \df Y_t=\partial_x a_\theta(X_t)Y_t\df t+\df \dif Z_t,\quad Y_0=0,
$$
and hence can be written explicitly: \begin{equation}\label{DX_exp}
 \dif X_t=\Ef_t
 \int_0^t
 \int_{\R}\Ef_s^{-1}\varrho(u) \nu(\df s, \df
u), \quad \Ef_t:=\exp\left\{\int_0^t\partial_xa_\theta(
X_\tau)\df \tau\right\}.
\end{equation}

The same argument gives  (we omit the detailed exposition):
\be\label{D_delta}
\dif
\delta(1)=\int_0^T\int_{\Re}\chi'(u)\varrho(u) \nu(\df s,\df u),
\ee
\be\label{Dr}
 \dif \Ef_t=\Ef_t
 \int_0^t\partial^2_{xx}a_\theta(X_\tau)\dif
 X_\tau\df \tau, \ee
\begin{equation}\label{DDX}\begin{aligned}
 \dif^2 X_t&:=\dif(\dif X_t)=\dif\Ef_t
\left(\int_0^t
 \int_{\R}\Ef_s^{-1}\varrho(u) \nu(\df s, \df
u)\right)
\\&+\Ef_t\int_0^t\int_{\Re}\left(\Ef_s^{-1}\varrho(u)\varrho'(u)-\Ef_s^{-2}\dif
\Ef_s\varrho(u)\right)\nu(\df s, \df u).
\end{aligned}
\end{equation}
The second derivative $\dif^2 X_t$ is stochastically
differentiable, as well;  a cumbersome explicit formula for
$\dif^3 X_t$,  analogous to (\ref{DDX}), is omitted. Note that the
assumption on $\sigma''$ from \textbf{H} (iii) is required to
bound $\chi'$ and prove (\ref{D_delta}), and the assumption on the
derivatives $\prt_xa, \prt_{xx}^2a, \prt^3_{xxx}a$ is used to get
the existence of the derivatives $\dif X_t, \dif^2 X_t, \dif^3
X_t$.

Recall that, although this is not given explicitly in the
notation, the solution  $X_t$ to Eq. (\ref{eq1}) is a function
which depends on the parameter $\theta$. Applying Theorem II.2.8.5
\cite{scorohod2} once more, we get that, under \textbf{H} (i) and
conditions of statement II of  Theorem \ref{mainthm1}, $X_t$  is
$L_2$-differentiable w.r.t. $\theta$ (in the sense similar to
(\ref{DX})),  and respective derivative equals
\begin{equation}\label{30}
\partial_\theta X_t=
 \Ef_t\int_0^t\Ef_s^{-1}[\partial_\theta a_\theta] (X_s)\df  s.
\ee This derivative is $L_2$-continuous w.r.t. $\theta$ and
stochastically  differentiable with
\begin{equation}\label{3}
\begin{aligned}
  \dif(\partial_\theta X_t)&=\dif\Ef_t \int_0^t
  \Ef_s^{-1}[\partial_\theta a_\theta] (X_s)\df  s\\
  &+  \Ef_t\int_0^t\left(\Ef^{-1}_s [\partial^2_{x\theta} a_\theta](X_s) (\dif X_s)
  -\Ef^{-2}_s\dif\Ef_s [\partial_\theta a_\theta] (X_s)\right)\df s.
\end{aligned}
\end{equation}

\subsection{Moment bounds}\label{s33} Here we collect several moment
bounds required in the consequent proofs.

First, we show that for every $p\geq 1$ there exists $C_{p}$ s.t.
\be\label{L_p}
  \E_x^\theta\Big|\dif X_t\Big|^p\leq C_p,  \quad x\in \Re, \theta\in \Theta, t\leq T.
\ee
 Note that  because $\prt_x a_\theta$ is bounded, there exist positive $C_1, C_2$
 such that
 \be\label{Ef_bound}
 C_1\leq \Ef_t\leq C_2, \quad t \in [0, T].
 \ee
Then (\ref{L_p}) would  follow from the same  bound for the It\^o
integral of the deterministic function $\varrho$ w.r.t. $\nu$. To
prove that bound, we apply  Lemma 5.1 \cite{Bichteler_Grav_Jacod}
with $L\equiv 1, \eta=\varrho$ and estimate separately the
integral over the compensator $\df t\ \mu(\df u)$. Because
$\prt^2_{xx}a$ is bounded, the analogue of (\ref{L_p}) holds true
for $\dif \Ef_t$ instead of $\dif X_t$; see (\ref{Dr}).  Repeating
(with minor modifications)  the same argument, we prove that the
analogue of (\ref{L_p}) holds true for
$$
\dif^2 X_t,\quad  \dif^2 \Ef_t, \quad  \dif^3 X_t.
$$
For instance, to get the analogue of (\ref{L_p}) for $\dif^2 X_t$
given by (\ref{DDX}), we use (\ref{Ef_bound}), the analogue of
(\ref{L_p}) for $\dif \Ef_t$, the H\"older inequality, and  Lemma
5.1 \cite{Bichteler_Grav_Jacod} with $L_s=\dif \Ef_s,
\eta=\varrho$.

Finally, by assumption \textbf{H} (iii) and the choice of
$\varrho$,  the function $$ \chi=-(\sigma'/\sigma)\varrho-\varrho'
$$
is bounded by $C|u|$  and vanishes for  $|u|\geq u_0$. Hence by
the  same argument the analogue of (\ref{L_p}) holds true for
$\delta(1)$, given by  (\ref{delta_1}).

Next,  note that by the condition \textbf{H} (i) we have $\E
|Z_t|^{2+\kappa}\leq C, $ $t\leq T$. Recall that $a_\theta(x)$ has
a linear growth bound w.r.t. $x$. Then the standard argument based
on the Gronwall lemma shows that \be\label{mom_kappa_Z}
\E_x^\theta|X_t|^{p}\leq C(1+|x|^p), \quad p\in [1, 2+\kappa). \ee
Similar bounds with $p<2+\kappa$ hold true for $\prt_\theta X_t$
and $\dif(\prt_\theta X_t)$. This follows from  formulae
(\ref{30}), (\ref{3}), the H\"older inequality, and $L_p$-bounds
(with arbitrarily large $p$) for $\dif X_t, \dif \Ef_t$.
Similarly,  $\dif X_t, \dif^2 X_t, \dif^3 X_t$ are
$L_2$-differentiable w.r.t. parameter $\theta$ (we omit the
details, but note that this is the place where we require the
assumption on $\prt^2_{x\theta}a$, $\prt^3_{xx\theta}a$,
$\prt^4_{xxx\theta}a$). Respective derivatives $\prt_\theta \dif
X_t, \prt_\theta\dif^2 X_t,$ $\prt_\theta\dif^3 X_t$ can be
written explicitly, and satisfy bounds analogous to
(\ref{mom_kappa_Z}).

Finally, we show that for every $p\geq 1, t\in (0, T]$  there
exists $C_{p,t}$ s.t. \be\label{moments} \E_x^\theta(\dif
X_t)^{-p}\leq C_{p,t}, \quad x\in \Re, \theta\in \Theta \ee
(recall that $\varrho\geq 0$ and therefore $\dif X_t$ is
non-negative). By  (\ref{Ef_bound}) and non-negativity of
$\varrho$,
$$
\dif X_t\geq (C_2/C_1) \int_0^t\int_{\Re}\varrho(u)
 {\nu}(\df s, \df u) \geq (C_2/C_1)\int_0^t\int_{|u|\leq u_0}u^2  {\nu}(\df s, \df u).
$$
Hence by assumption \textbf{H} (iv)
$$\begin{aligned}
\P_x^\theta(\dif X_t<\eps)&\leq
\P_x^\theta\left(\int_0^t\int_{(C_1\eps /C_2)^{1/2}\leq|u|\leq
u_0}u^2  {\nu}(\df s, \df
u)=0\right)\\&=\exp\left[-t\mu\Big(u:\{(C_1\eps
/C_2)^{1/2}\leq|u|\leq u_0\}\Big)\right]=o(\eps^p), \quad \eps \to
0
\end{aligned}
$$
for every $p\geq 1$, which implies the required statement.

\section{Proof of Theorem \ref{mainthm1}}

Our first  step is to prove that $(\dif X_t)^{-1}\in
\mathrm{dom}(\delta)$ and \be\label{Xi} \Xi_t:=\delta\left({1\over
\dif X_t}\right)={\delta(1)\over \dif X_t}+{\dif^2 X_t\over (\dif
X_t)^2}. \ee To do that, for arbitrary $\eps>0$ we consider the
variable $G_{t,\eps}=(\dif X_t+\eps)^{-1}$.   By statement 1 of
Proposition \ref{lem01},  applied to $F=\dif X_t$ and $\phi\in
C^1_b(\Re)$ such that $\phi(x)=1/x, x\geq\eps$, the variable
$G_{t,\eps}$ has the stochastic derivative  $-{\dif^2 X_t/(\dif
X_t+\eps)^2}$. Then by the statement 3 of Proposition \ref{lem01}
\be\label{Xi_eps} \Xi_{t,
\eps}=\delta\left(G_{t,\eps}\right)={\delta(1)\over \dif
X_t+\eps}+{\dif^2 X_t\over (\dif X_t+\eps)^2}. \ee By (\ref{L_p})
and (\ref{moments}), $G_{t,\eps}\to (\dif X_t)^{-1}$ in $L_2$.
Using, in addition,  analogues of (\ref{L_p}) for $\delta(1)$ and
$\dif^2 X_t$, we see that the right hand side term in
(\ref{Xi_eps}) converges to that in (\ref{Xi}) in $L_2$ as
$\eps\to 0$. Hence (\ref{Xi}) holds true because  $\delta$ is a
closed operator.

Now we can finalize the proof of \emph{statement I}; the argument
here  is quite analogous to the one from the proof of Proposition
3.1.2 in  \cite{nualart}, hence we omit details. By Proposition
\ref{lem01}, the definition of $\delta=\dif^*$, and  (\ref{Xi}),
for every $\varphi\in C^1_b(\Re)$ we have
$$
\E_x^\theta\varphi'(X_t)=\E_x^\theta\dif(\varphi(X_t))\left({1\over\dif
X_t}\right)=\E_x^\theta \varphi(X_t)\delta \left({1\over \dif
X_t}\right)= \E_x^\theta \varphi(X_t)\Xi_t.
$$
Approximating $\varphi_y:=\1_{[0,\infty)}(\cdot-y)$ by a sequence of
$\varphi_n\in C^1_b(\Re)$, we get the representation (\ref{irp}).

By Theorem II.2.8.3 \cite{scorohod2} applied to Eq. (\ref{eq1}),
$X_t$  depend continuously (in $L_2$) w.r.t. parameters $x, t,
\theta$. The same argument gives the
continuity (in $L_2$) of $\dif X_t$ and $\dif^2 X_t$
 w.r.t. parameters $x, t, \theta$; note that both these derivatives
 can be defined as  solutions to certain SDE's, hence one can apply
 Theorem II.2.8.3 \cite{scorohod2} iteratively. Then it is easy to
 show that $\Xi_t$ depend continuously (in $L_2$) on the parameters
 $x, t, \theta$, as well. Indeed, by the continuity of $\dif X_t, \dif^2 X_t$
 and non-negativity of $\dif X_t$, for every $\eps>0$ the functional $\Xi_{t,\eps}$,
 defined by (\ref{Xi_eps}), is continuous. Using the moment
 bounds (\ref{L_p}) and  (\ref{moments}) it is easy to show that $\Xi_{t,\eps}$
 converges to $\Xi_t$ in $L_2$ as $\eps\to \infty$ uniformly in some neighborhood
 of any given point $(x, t, \theta)\in \Re\times (0, \infty)\times \Theta$,
 hence the limiting functional $\Xi_t$ depend continuously on $x, t, \theta$.

By representation (\ref{irp}), we have \be\label{null}
\P_x^\theta(X_t=y)=0, \quad x, y\in \Re, \quad t>0, \quad
\theta\in \Theta. \ee Then $L_2$-continuity of $\Xi_t$ and
representation (\ref{irp}) provide that $p_t^\theta(x,y)$ is
continuous w.r.t. $(t,x,y, \theta)$.

To prove \emph{statement II}, we make one more integration by
parts in the right hand side of (\ref{irp}). Similarly to the
proof of (\ref{Xi}),  we use Proposition \ref{lem01} and
(\ref{L_p}), (\ref{moments}) to show that $\Xi_t/(\dif X_t)$
belongs to $\mathrm{dom}(\delta)$ with \be\label{delta_Xi}
\delta\left({\Xi_t\over \dif X_t}\right)={(\delta(1))^2-\dif
\delta(1)\over (\dif X_t)^2}+{3\delta(1)\dif^2X_t-\dif^3 X_t\over
(\dif X_t)^3}+{3(\dif^2 X_t)^2\over (\dif X_t)^4}. \ee

By  (\ref{irp}), we have \be\label{irp_1}
p_t^\theta(x,y)=\E_x^\theta \psi_y(X_t) \delta\left({\Xi_t\over
\dif X_t}\right), \ee where $\psi_y=(\cdot-y)\vee 0$ is an
absolutely continuous function with  the derivative equal to
$\varphi_y$. Recall that  $X_t$ is $L_2$-differentiable w.r.t.
parameter $\theta$, see (\ref{30}) for its derivative. In
addition, $\dif X_t, \dif^2 X_t,$ and  $\dif^3 X_t$, are
$L_2$-differentiable w.r.t. $\theta$, and all these derivatives
satisfy moment bounds similar to (\ref{mom_kappa_Z}). Now it is
easy to prove that $\delta\left(\Xi_t/(\dif X_t)\right)$ is
$L_2$-differentiable w.r.t. $\theta$ (the explicit formula of the
derivative is omitted). One can just replace $\dif X_t$ in the
denominator in the formula  (\ref{delta_Xi}) by $\dif X_t+\eps$,
prove that this new functional is $L_2$-differentiable w.r.t.
$\theta$ using the chain rule, and then show using (\ref{moments})
that both this functional and its derivative w.r.t. $\theta$
converge (locally uniformly) in $L_2$ as $\eps\to 0$,
respectively,  to $\delta\left(\Xi_t/(\dif X_t)\right)$ and to the
functional $\prt_\theta\delta\left(\Xi_t/(\dif X_t)\right)$ which
comes from the formal differentiation of (\ref{delta_Xi}). This
argument also shows that $\delta\left(\Xi_t/(\dif X_t)\right)$ and
$\prt_\theta\delta\left(\Xi_t/(\dif X_t)\right)$ depend
continuously (in $L_2$) on $x, t, \theta$. Therefore, we can take
a derivative at the right hand side in (\ref{irp_1}), which gives
$$
\prt_\theta
p_t^\theta(x,y)=\E_x^\theta\left[\varphi_y(X_t)\prt_\theta X_t
\delta\left({\Xi_t\over \dif X_t}\right)+\psi_y(X_t)
\prt_\theta\delta\left({\Xi_t\over \dif X_t}\right)\right].
$$
This function is continuous w.r.t. $(t,x, y,\theta)$ because
$X_t, \prt_\theta X_t, \delta\left(\Xi_t/(\dif X_t)\right)$, and
$\prt_\theta\delta\left(\Xi_t/(\dif X_t)\right)$ depend
continuously (in $L_2$) on $x, t, \theta$, and (\ref{null}) holds
true.

To prove \emph{statement III}, we use moment bounds for
$\prt_\theta X_t$, $\dif(\prt_\theta X_t)$, $\dif X_t$, $\dif
X_t^2$ to get,   similarly to the proof of (\ref{Xi}), that
$\prt_\theta X_t/(\dif X_t)$ belongs to $\mathrm{dom}(\delta)$ and
\be\label{Xi_1} \Xi_t^1:=\delta\left({\prt_\theta X_t\over \dif
X_t}\right)={(\prt_\theta X_t) \delta(1)\over \dif
X_t}+{(\prt_\theta X_t) \dif^2 X_t \over (\dif
X_t)^2}-{\dif(\prt_\theta X_t)\over \dif X_t}. \ee Then for any
test function  $f\in C^1(\Re)$ with a bounded derivative we have
\begin{multline}\label{rightside}
    \partial_\theta\E_x^\theta f(X_t)=
    \E_x^\theta f'(X_t)(\partial_\theta X_t)=
    \E_x^\theta \dif f(X_t)\left({\prt_\theta X_t\over \dif X_t}\right)\\
    =\E_x^\theta f(X_t)\Xi^1_t=\E_x^\theta f(X_t)g_t^\theta(x,X_t);
\end{multline}
see (\ref{g}) for the definition of $g_t^\theta(x,y)$. Because the
test  function  $f$ is arbitrary, the integral identity
(\ref{rightside}) proves (\ref{d_theta_rep}).

\section{Proof of Theorem \ref{mainthm2}}

First, we formulate some properties of   $p_t^\theta$  and
$g_t^\theta=\prt_\theta \log p_t^\theta$, which follows from the
integral representations for these functions and moment bounds
obtained above.

\begin{lem}\label{lem_aux}  For every $p<2+\kappa$ there
exists constant $C$ which depends on $t$ and $p$ only, such that
\be\label{p_tails}
p_t^\theta(x,y)\leq C(1+|x-y|)^{-p},
\ee
\be\label{g_mom}
\E_x^\theta \Big|g_t^\theta(x, X_t)\Big|^p\leq C(1+|x|)^p.
\ee
\end{lem}
\begin{proof}
By the moment bounds from Section \ref{s33} and formulae
(\ref{Xi}), (\ref{Xi_1}), we have \be\label{Xi_mom} \E_x^\theta
\Big|\Xi_t\Big|^{p'}\leq C \ee for every $p'\geq 1$, and
\be\label{Xi_mom_1} \E_x^\theta|\Xi_t^1|^{p}\leq C(1+|x|^p) \ee
for every $p\in [1, 2+\kappa)$, with the constants $C$ depending
on $t, p'$ (resp. $p$), only.

Because
$$
g_t^\theta(x, X_t)=\E_x^\theta\Big[\Xi_t^1\Big|X_t\Big],
$$
inequality (\ref{g_mom}) follows directly from (\ref{Xi_mom_1})
and Jensen's  inequality. To get (\ref{p_tails}), we use the
standard argument, e.g.  \cite{nualart}, Lemma 3.1.3 and Example
afterwards. By the representation (\ref{irp}) and the H\"older
inequality,
$$
p_t^\theta(x,y)\leq C\Big(\P_x^\theta(X_t>y)\Big)^{(p'-1)/p'}.
$$
Recall that $Z_t$ has finite moment of the order $2+\kappa$ and
$a_\theta$ has  bounded derivative in $x$. Then by the Gronwall
lemma
$$
\E_x^\theta|X_t-x|^{2+\kappa}\leq C.
$$
Then for $y>x$
$$
\P_x^\theta(X_t>y)\leq \min\{1, C|x-y|^{-(2+\kappa)}\},
$$
which gives (\ref{p_tails}) with $p=(p'-1)(2+\kappa)/p'$; the
latter  value can be made arbitrarily close to $2+\kappa$ by the
choice of $p'$.  For $y<x$ one should use instead of (\ref{irp})
the representation
$$
p_t^\theta(x,y)=-\E_x^\theta\left[\Xi_t \1_{X_t\leq y}\right],
$$
which is equivalent to (\ref{irp}), because $\Xi_t$ is a
stochastic  integral and therefore has zero expectation.

\end{proof}

Let us proceed  with the proof of the Theorem. Formula
(\ref{rep_p})  and statement I of Theorem \ref{mainthm1}
immediately provide the  continuity property (a) from the
definition of a regular statistical experiment.  To prove the
$L_2$-differentiability property (b), and $L_2$-continuity of the
derivative (c),  we put for $\ve>0$
$$
\psi_\ve(z)=\begin{cases}0,&z<\ve/2,\\ {(z-\ve/2)^2\over
2\ve^{3/2}},&z\in [\ve/2,\ve],\\ \sqrt{z}-{7\sqrt{\ve}\over
8},&z\geq \ve.\end{cases}
$$
By the construction, $\psi_\ve\in C^1$ and $\psi_\ve(z)=0$ for
$z\leq \eps/2$.  Then, by statement II of Theorem \ref{mainthm1}
and statement 2 of Lemma \ref{lem_aux} the mapping $\theta\mapsto
\eta_\ve^\theta:=\psi_\ve(p^\theta)\in L_2(\Re^n, \lambda)$ is
continuously differentiable with the derivative equal
$$
\zeta_\ve^\theta:= {\df\over
\df\theta}\eta_\ve^\theta=\psi_\ve'(p^\theta)\prt_\theta p^\theta;
$$
see (\ref{rep_dp}) for the formula for $\prt_\theta p^\theta$. By the construction,
$$
\psi_\ve(z)\to \psi_0(z):=\sqrt{z}, \quad \psi_\ve'(z)\to
\psi_0'(z)= {1\over 2\sqrt{z}}, \quad \ve\to 0.
$$
Hence to prove properties (b), (c) it is enough to show that
\be\label{conv} \eta_\ve^\theta\to
\eta_0^\theta:=\psi_0(p^\theta),\quad \zeta_\ve^\theta\to
\zeta_0^\theta:=\psi_0'(p^\theta)\prt_\theta p^\theta\quad
\hbox{in}\quad L_2(\Re^n, \lambda) \ee uniformly by $\theta$. We
show the second convergence in (\ref{conv}), the proof of the
first one is similar and simpler. By the explicit form of
$\psi'_\ve$ and the H\"older inequality,
$$\begin{aligned}
\int_{\Re^n}(\zeta_\ve^\theta- \zeta_0^\theta)^2\df\lambda&\leq
{1\over 4}\int_{p^\theta\leq \ve} (g^\theta)^2p^\theta
\df\lambda\leq {1\over 4}\left(\int_{\Re^n} |g^\theta|^pp^\theta
\df\lambda\right)^{\frac{2}{p}}\left(\int_{p^\theta\leq \ve}
p^\theta \df\lambda\right)^{\frac{p-2}{p}}
\end{aligned}
$$
for $p\geq 2$. Take $p\in (2, 2+\kappa)$, then by Jensen's
inequality,  representation (\ref{rep_dp}),  and
(\ref{mom_kappa_Z}),
$$\begin{aligned}
\int_{\Re^n} &|g^\theta|^pp^\theta \df\lambda=\E_{x}^\theta
\left|\sum_{k=1}^ng_{t_k-t_{k-1}}^\theta(X_{t_{k-1}},
X_{t_{k}})\right|^p\\&\leq n^{p-1}\sum_{k=1}^n\E_{x}^\theta
\left|g_{t_k-t_{k-1}}^\theta(X_{t_{k-1}}, X_{t_{k}})\right|^p\leq
n^{p-1} C\sum_{k=1}^n\E_{x}^\theta\Big|1 + |X_{t_{k-1}}|\Big|^p\leq
\tilde C,
\end{aligned}
$$
where  constant $\tilde C$ does not depend on $\theta$. On the
other hand,
$$
\int_{p^\theta\leq \ve} p^\theta \df\lambda\leq
\sqrt{\eps}\int_{\Re^n}\sqrt{p^\theta}\df\lambda,
$$
and by (\ref{p_tails}) with  $p\in (2, 2+\kappa)$
$$\int_{\Re^n}\sqrt{p^\theta}\df\lambda\leq
C\int_{\Re^n}\prod_{k=1}^n\Big(1+|x_{k-1}-x_k|\Big)^{-p/2}dx_1\dots dx_n\leq \hat C$$
with  constant $\hat C$ which does not depend on $\theta$.
Summarizing all the above, we get the second convergence in
(\ref{conv}), uniform by $\theta$. This completes the proof of the
regularity. The formula for the Fisher information follows from
the identity
$$
\int_{\Re^n}
\left(\prt_\theta\sqrt{p^\theta}\right)^2\df\lambda={1\over
4}\int_{\Re^n} (g^\theta)^2p^\theta \df\lambda={1\over
4}\E_{x}^\theta
\left(\sum_{k=1}^ng_{t_k-t_{k-1}}^\theta(X_{t_{k-1}},
X_{t_{k}})\right)^2,
$$
Markov property, and  the observation
that
$$
\E_x^\theta g_t^\theta(x, X_t)=0
$$
for every $x\in \Re, \theta\in \Theta, t>0$, which follows  from
(\ref{rightside}) with $f\equiv 1$.

\section*{Acknowledgements} The authors express their
deep gratitude to  the referees for  numerous  comments  and
suggestions that  were helpful and led for substantial improvement
of  the paper.


\begin{thebibliography}{}

\bibitem{BC} V. Bally, E. Cl\'ement. Integration by parts
formula and applications to equations  with jumps. \emph{Probab.
Theory Relat. Fields}  151, 3-4, 613 - 657, 2011.

\bibitem{Bichteler_Grav_Jacod} K. Bichteler,  J.-B. Gravereaux,
J. Jacod. \textit{Malliavin calculus for processes with jumps}.
Gordon and Breach science publishers, N.Y., London, Paris, Tokyo,
1987.

\bibitem{Bismut_jumps} J.-M. Bismut. Calcul des variations stochastique
et processus de sauts. \textit{Z. Warw. theor. verw. Geb.},
56(4):469-505, 1981.

\bibitem{BD}  N. Bouleau,
L. Denis.  Application of the lent particle method to
Poisson-driven  SDE's \emph{Probab. Theory Relat. Fields}  151,
3-4, 403 - 433, 2011.

\bibitem{carr}
P. Carr, H. Geman, D.B. Madan, M. Yor. Stochastic volatility for
L\'{e}vy processes. \textit{Math. Finance},  13:345–382, 2003.

\bibitem{bridges} L. Chaumont,  G. Uribe Bravo.
Markovian bridges: Weak continuity and pathwise constructions.
\textit{Ann. Probab.}, 39(2):609-647, 2011.

\bibitem{Corc}
J.M. Corcuera, A. Kohatsu-Higa. Statistical inference and
Malliavin calculus. \textit{Seminar on Stochastic Analysis, Random
Fields and Applications VI, Progress in Probability, Springer
Basel}, 63:59-82, 2011.

\bibitem{finance} E. Fourni\'e, J.M. Lasry, J. Lebuchoux, P.-L. Lions and N.
Touzi. Some applications of Malliavin calculus to Monte Carlo
methods in finance. \textit{Finance and Stochastics}, 3:391-412,
1999.

\bibitem{scorohod2} I.I. Gikhman, and  A.V. Skorokhod.
\emph{Stochastic differential equations and their applications.}  New York,
Springer-Verlag, 1972.


\bibitem{Gobet1} E. Gobet. Local asymptotic mixed normality property for
elliptic diffusion: a Malliavin calculus approach. \emph{Bernoulli}
7(6), 899 -- 912, 2001.

\bibitem{Gobet2} E. Gobet. LAN property for ergodic diffusions with
discrete observations. \emph{Ann. I. H. Poincar\'e – PR} 38, 5  711 -- 737, 2001.


\bibitem{IKh}  I.A. Ibragimov, R.Z. Hasminskii.
\textit{Statistical estimation: asymptotic theory}. New York,
Springer-Verlag, 1981.

\bibitem{Ivanenko} D.O. Ivanenko. LAN property for  solutions
of L\'evy driven SDE's. (in preparation)

\bibitem{IvanenkoKulik} D.O. Ivanenko, A.M. Kulik. Asymptotic
properties of the MLE for  solutions of L\'evy driven SDE's. (in preparation).


\bibitem{novikov}
E.A. Novikov. Infinitely divisible distributions in turbulence.
\textit{Phys. Rev. E}, 50:R3303–R3305, 1994.

\bibitem{nualart}
D. Nualart. Analysis in Wiener space and anticipating stochastic
calculus. \textit{Springer Berlin / Heidelberg, Lecture Notes in
Mathematics}, 1690:123-220, 1998.

\bibitem{Ros1}
J. Rosi\'{n}ski. Tempering stable processes. \textit{Stoch. Processes  and
 Appl.},  117(6):677-707, 2007.

\bibitem{gorod}
B. Simon, M. Read. \textit{Methods of Modern Mathematical
Physics}. Functional Analysis, Academic Press, San Diego, 1972.

\bibitem{t_simon_support} T. Simon. Support theorem for jump
processes. \textit{Stoch. Proc. and Appl.}, 89:1-30, 2000.

\bibitem{skorohod}
 A. V. Skorokhod. On differentiability of the measures which
 correspond to stochastic processes. I.  Processes with
 independent increments. \emph{Theory of Probability and Appl.}  2, no. 4, 407-432, 1957.
\end{thebibliography}
\end{document}